\documentclass[a4paper,11pt]{amsart}
\usepackage{amsmath,amssymb,amsfonts,amsthm,exscale,calc}
\usepackage[latin1]{inputenc}
\usepackage{latexsym,textcomp}
\usepackage{graphicx,epsf}
\usepackage{dsfont}
\usepackage[german, english]{babel}
\usepackage{hyperref}

\newtheorem{lemma}{Lemma}[section]
\newtheorem{theorem}[lemma]{Theorem}

\newtheorem{proposition}[lemma]{Proposition}
\newtheorem{corollary}[lemma]{Corollary}

\theoremstyle{definition}
\newtheorem{definition}[lemma]{Definition}
\newtheorem{example}[lemma]{Example}
\newtheorem*{remark}{Remark}

\numberwithin{equation}{section}
\newcommand{\comment}[1]{}

\newcommand{\Z}{{\mathbb Z}}
\newcommand{\R}{{\mathbb R}}

\newcommand{\N}{{\mathbb N}}

\newcommand{\A}{{\mathcal{A}}}
\newcommand{\LL}{{\mathcal{L}}}

\newcommand{\rH}{{\partial_h X}}
\newcommand{\rR}{{\partial_R X}}

\newcommand{\al}{{\alpha}}

\newcommand{\de}{{\delta}}

\newcommand{\ka}{{\kappa}}
\newcommand{\ph}{{\varphi}}
\newcommand{\lm}{{\lambda}}

\newcommand{\as}[1]{\left\langle #1\right\rangle}

\newcommand{\aV}[1]{\left\Vert #1\right\Vert}

\newcommand{\ow}[1]{\widetilde{ #1}}

\newcommand{\Hm}[1]{\leavevmode{\marginpar{\tiny%
$\hbox to 0mm{\hspace*{-0.5mm}$\leftarrow$\hss}%
\vcenter{\vrule depth 0.1mm height 0.1mm width \the\marginparwidth}%
\hbox to 0mm{\hss$\rightarrow$\hspace*{-0.5mm}}$\\\relax\raggedright
#1}}}

\begin{document}

\title[Uniformly transient graphs]{Note on uniformly transient  graphs}

\author[Keller]{Matthias Keller}
\address{M. Keller, Mathematisches Institut \\Friedrich Schiller Universit{\"a}t Jena \\07743 Jena, Germany }\email{m.keller@uni-jena.de}

\author[Lenz]{Daniel Lenz}
\address{D. Lenz, Mathematisches Institut \\Friedrich Schiller Universit{\"a}t Jena \\07743 Jena, Germany } \email{daniel.lenz@uni-jena.de}

\author[Schmidt]{Marcel Schmidt}
\address{M. Schmidt, Mathematisches Institut \\Friedrich Schiller Universit{\"a}t Jena \\07743 Jena, Germany } \email{schmidt.marcel@uni-jena.de}

\author[Wojciechowski]{Rados{\l}aw K. Wojciechowski}
\address{R. K. Wojciechowski
\\ The Graduate Center of the City University of New York \\ 365 Fifth Avenue \\ New York, NY 10016}
\address{ York College of the City University of New York \\ Jamaica, NY 11451 \\ USA  } \email{rwojciechowski@gc.cuny.edu}

\begin{abstract}
We study a special class of graphs with a strong transience feature
called uniform transience. We characterize uniform transience via a
Feller-type property and via validity of an isoperimetric
inequality. We then give a further characterization via  equality
of the Royden boundary and the harmonic boundary and  show that the
Dirichlet problem has a unique solution for such graphs. The Markov
semigroups and resolvents (with Dirichlet boundary conditions) on
these graphs are shown to be ultracontractive. Moreover, if the
underlying measure is finite, the semigroups and resolvents are
trace class and their generators have $\ell^p$ independent pure
point spectra (for $1 \leq p \leq \infty$).

Examples of uniformly transient graphs include Cayley graphs of
hyperbolic groups as well as trees and Euclidean lattices of
dimension  at least three. As a surprising consequence, the Royden
compactification of such lattices turns out to be the one-point
compacitifcation and the Laplacians of such lattices have pure point
spectrum if the underlying measure is chosen to be finite.
\end{abstract}


\maketitle



\section*{Introduction}
Spectral geometry  is concerned with the interplay of spectral
theory of Laplacians and the geometry of the underlying structure.
The two basic paradigms are given by  Riemannian manifolds and
graphs.  There are many similarities between the case of graphs and
the case of manifolds. Indeed, a common framework is provided within
the theory of Dirichlet forms. Still, there are also crucial
differences. Structurally, a main difference is that manifolds give
local Dirichlet forms whereas graphs give non-local Dirichlet forms.

\smallskip

As far as examples are concerned, there is also another difference:
In the case of manifolds, the Riemannian structure provides both the
Laplacian on smooth functions and a canonical measure. The
situation on graphs is rather  different. One is given two pieces of
data on a countably infinite set $X$  viz - in notation explained
below in Section \ref{section-setup} -
\begin{itemize}
\item  a graph structure $(b,c)$ and
\item a measure $m$
\end{itemize}
and these two pieces of data are completely \textit{independent}. In
this sense, there are more parameters available in the case of
graphs.

One way to deal with the abundance of possible measures  in the
graph case is to restrict attention to certain special measures. In
this context, two choices have attracted particular attention. One
is the measure derived from $b$ by taking $m$ to be the vertex
degree. The corresponding Laplacian is known as the normalized
Laplacian.  The other choice  is the constant measure. Both of these
choices  have their merits. Indeed, it seems that, for a long time,
the study of Laplacians on graphs was restricted to one of these two
choices. In particular, the spectral geometry of normalized
Laplacian has been quite a focus of attention,  see e.g.
\cite{Ak,BaJ,BHJ,BJ,CGY,HJ,HuJ,JL,Liu} and references therein as
well as \cite{HJ,PR}  for higher order Laplacians.

Recently, however, there has been an  outburst of all sorts of
studies of Laplacians on graphs with  general measures, see e.g.
\cite{ATH,BHK,BHY,CdVTHT, Fol1,Fol2,FLW,FP,Gol,GHM,Gue,HKLW,
HKW,HKMW,Hua1,KL1,KL2,KLVW,Mi,MT,Mug,Woj4} and references therein.
In some sense, a comparable development can be seen in the study of
manifolds. There, weighted manifolds have become a focus of
attention in certain questions of spectral geometric nature,  see
e.g. \cite{Gri, GM}.

Given  this situation,  there is  substantial interest in features
of the graph which do NOT depend on the choice of the measure.

One such feature is transience / recurrence. Another feature is
compactness of the underlying structure. Indeed, quite recently, the
concept of a  {canonically compactifiable graph} has been brought
forward  in \cite{GHKLW}. Canonically compactifiable graphs have
many  claims to  model a (relatively) compact situation.

Here, we present another property which is independent of the
measure. This property is stronger than transience and weaker than
canonical compactifiability.  There are various ways to look at this
property. Indeed, the main abstract result of this note (Theorem
\ref{main}) shows that it  can simultaneously  be seen as a strong
transience condition or as a strong Feller-type condition or as the
validity of a strong isoperimetric inequality. We call it
\textit{uniform transience}. This property has already appeared
in the literature in  several places, see e.g.
\cite{BCG,Win} in, yet, other manifestations. A systematic treatment
- as given below -  is still missing until now.

As discussed below, the class of uniformly transient graphs contains
all  non-trivial trees  and all Cayley graphs of hyperbolic groups
(with standard weights) as well as all transient graphs with a
quasitransitive automorphism group.  In particular, all Euclidean
lattices $\Z^d$ for $d\geq3$ fall into this category.

Uniform transience  has  a  certain compactness flavor to it.  In
fact, every canonically compactifiable graph is uniformly transient
(Corollary \ref{cc-implies-ut}). Thus, all models considered in
\cite{GHKLW} fall into our framework here. Moreover, it is possible
to characterize uniform transience by a boundedness condition with
respect to a certain metric (Theorem \ref{metric-criterion}).
Furthermore, it is possible to characterize uniform transience  via
the Royden boundary. As a consequence, we can show unique
solvability of the Dirichlet problem for uniformly transient graphs.
This is discussed in Section \ref{Royden}. The methods developed in
Section \ref{Royden} can be extended to reprove the (well-known)
existence of solutions for the Dirichlet problem on general graphs.
We include a discussion in Section \ref{DP-general-graphs}.

As mentioned already, all  canonically compactifiable graphs are
uniformly transient. We can even   characterize the canonically
compactifiable graphs within the class of uniformly transient graphs
as those which for which all harmonic functions of finite energy are bounded
(Theorem \ref{cc-within-ut}). In this context, we can also prove
that for a transient graph the Royden compactification agrees with
the one-point compactification if and only if the graph is uniformly
transient and has the Liouville property (Corollary
\ref{Liouville-one-point}). As a particular class of examples for
this we discuss Euclidean lattices.

Uniformly transient graphs yield ultracontractive semigroups
independently of the underlying measure (Lemma
\ref{t:ultracontractive}). This can then be used to show that they
yield pure point spectrum with $\ell^p$ independent spectrum
whenever the underlying measure is finite (Theorem
\ref{theorem-spectral}).

Our abstract  results give remarkable and  somewhat surprising
consequences for the  Euclidean lattices  $\Z^d$ for $d\geq 3$.
These can easily be seen to be uniformly transient. From this, we
then obtain that the Royden compactification of such a lattice is
the one-point compactification. This is in sharp contrast to the
case of smaller dimensions. In fact, the Royden compactification of
the one-dimensional lattice is an enormous object (see \cite{Wys}).
Moreover, we infer that the Laplacian on such a Euclidean lattice
has pure point spectrum whenever the lattice is equipped with a
finite measure. Details are discussed in the last two sections.


Our considerations make use of a certain  characterization of the
domain of the Laplacian with Dirichlet boundary condition and of a
certain characterization of transience. Both of these
characterizations are probably well-known. As we have not been able
to find them in the literature, we have included corresponding
discussions in one appendix each.  We also include an appendix discussing
the relation between harmonic functions and bounded harmonic functions.

\bigskip

\textbf{Acknowledgments.}  R.K.W. gratefully acknowledges financial
support provided by the Simons Foundation Collaboration Grant for
Mathematicians and PSC-CUNY Awards, jointly funded by the
Professional Staff Congress and the City University of New York.
M.S. has been financially supported by the Graduiertenkolleg 1523/2
: Quantum and gravitational fields. The work of D.L. and M.K. is
partially supported by the German Research Foundation (DFG).

\section{Framework: graphs, forms and Laplacians}\label{section-setup} In this section we introduce the
key objects of our study. These are forms on graphs and the
associated semigroups and Laplacians.  A convenient framework has
recently been  presented  in  \cite{KL1,KL2}. Here we follow these
works and refer to them  for further details and references.

\bigskip

Let $X$ be a countably infinite  set. The vector space of all
real-valued functions on $X$ is denoted by $C(X)$. The subspace of
all real-valued functions vanishing outside of a finite set is
denoted by $C_c (X)$ and the closure of $C_c (X)$ with respect to
the \textit{supremum norm}
$$\|f\|_\infty :=\sup_{x\in X} |f(x)|$$
is denoted by $C_0 (X)$. It is a complete normed space when equipped
with the supremum norm.

\smallskip

A \emph{graph} over $X$ is a pair $(b,c)$ such that
$b:X\times X\longrightarrow[0,\infty)$ is symmetric, has zero
diagonal, and satisfies
\begin{align*}
    \sum_{y\in X}b(x,y)<\infty
\end{align*}
for all $x\in X$ and $c:X\longrightarrow[0,\infty)$ is arbitrary.
Then,   $X$ is called  the \emph{vertex set}, $b$ the \emph{edge
weight} and $c$ the \emph{killing term} or \emph{potential}.
Elements $x,y\in X$ are said to be  \emph{neighbors} or
\emph{connected} by an edge of weight $b(x,y)$ if $b(x,y)>0$. If
the number of neighbors of each vertex is finite, then we call
$(b,c)$ or $b$ \emph{locally finite}. A finite sequence
$(x_{0},\ldots,x_{n})$ of pairwise distinct vertices such that
$b(x_{i-1},x_{i})>0$ for $i=1,\ldots,n$ is called a \emph{path} from
$x_0$ to $x_n$. We say that $(b,c)$ or $b$ is \emph{connected} if,
for every two distinct vertices $x,y\in X$, there is a path from $x$
to $y$.

\smallskip

Given a weighted graph $(b,c)$ over  $X$ we define the
\emph{generalized form} $\ow Q: C(X)\longrightarrow[0,\infty]$ by
\begin{align*}
\ow Q(f):=\frac{1}{2}\sum_{x,y\in X}b(x,y)|f(x)-f(y)|^{2}+\sum_{x\in X}c(x)|f(x)|^{2}
\end{align*}
and define the \emph{generalized form domain} by
\begin{align*}
    \ow D:=\{f\in C(X) :  \ow Q(f)<\infty\}.
\end{align*}
Functions in $\ow D$ are said to have \emph{finite energy}.

Clearly,  $C_{c}(X)\subseteq \ow D$  holds as $b(x,\cdot)$ is
summable for every $x\in X$. By Fatou's lemma, $\ow{Q}$ is lower semi-continuous
with respect to pointwise convergence. The form $\ow Q$ gives rise to a
semi-scalar product on $\ow D$ via
\begin{align*}
 \ow Q(f,g)=\frac{1}{2}\sum_{x,y\in X} b(x,y)(f(x)-f(y))(g(x)-g(y))+\sum_{x\in X}c(x)f(x)g(x).
\end{align*}
If $c\not\equiv 0$ and $b$ is connected, the form $\ow Q$ defines a
scalar product. Furthermore, the form $\ow{Q}$ satisfies a certain cut-off property.
Namely, for each {\em normal contraction} $C:\R\to \R$ (i.e., $C$ satisfies $|C(x) - C(y)| \leq |x- y|$
and $|C(z)| \leq |z|$ for arbitrary $x,y,z \in \R)$ and each $f \in C(X)$ we have
$$\ow{Q}(C \circ f) \leq \ow{Q}(f).$$

We will need the  following  well-known lemma (see e.g.
\cite{GHKLW}).

\begin{lemma}\label{l:DLip} Let $ (b,c)$ be a connected graph over $X$.
Then, for any $x,y\in X$, there exists  $d(x,y)\geq 0$ such that
for any $f\in \ow D$
\begin{align*}
|f(x)-f(y)|^{2}  \leq \ow Q(f) d(x,y).
\end{align*}
\end{lemma}
We  now choose a vertex $o\in X$ and define a semi-scalar product $\as{\cdot,\cdot}_{o}$ on $\ow D$ by
\begin{align*}
\as{f,g}_{o}=\ow Q(f,g)+f(o)g(o),
\end{align*}
for $f,g\in \ow D$ and the corresponding semi-norm
\begin{align*}
    \aV{f}_{o}:={\as{f,f}_{o}}^{1/2}=
    {(\ow Q(f)+|f(o)|^{2})}^{1/2}.
\end{align*}
If $b$ is connected, then $\as{\cdot,\cdot}_{o}$ defines a scalar
product and $\aV{\cdot}_{o}$ defines a norm on $\ow D$. The space
$\ow D$ has received a lot of interest since first being studied in
\cite{Yam}. A systematic investigation was given in the work of  Soardi
\cite{Soa}. In this context, we also   recall the following
well-known lemma which follows from Lemma \ref{l:DLip} above.

\begin{lemma}\label{l:pointevaluation} If $(b,c)$ is connected, then the point evaluation map
\begin{align*}
    \de_{x}:(\ow D,\aV{\cdot}_{o})\longrightarrow \R, \quad u\mapsto u(x),
\end{align*}
is continuous for each $x\in X$.
\end{lemma}

As the choice of $o\in X$ in the previous lemma is arbitrary, we
directly obtain the following consequence.

\begin{lemma}\label{l:equivalence} Let $(b,c)$ be a connected graph over $X$ and let $o_1, o_2\in X$ be
arbitrary. Then, the norms $\aV{\cdot}_{o_1}$ and $\aV{\cdot}_{o_2}$
are equivalent.
\end{lemma}
\begin{remark} In general $\aV{\cdot}_o$ and $\ow Q^{1/2}$ are not
equivalent norms on $C_c (X)$. In fact, they can be shown to be
equivalent if and only if the underlying graph is transient (see
Appendix \ref{Transience}).
\end{remark}
Our main focus of interest is a special subspace of $\ow D$. It is introduced next.
\begin{definition}[The space $\ow D_0$]
\label{def:Dzero} Let $(b,c)$ be a connected graph over $X$ and let
$o\in X$ be fixed.  Define $\ow D_0$ to be   the closure of $C_c
(X)$ in $\ow D$ with respect to $\aV{\cdot}_o$.
\end{definition}

\begin{remark} We think of the elements of $\ow D_0$ as functions
satisfying  ``Dirichlet boundary conditions at infinity.'' As is
clear from Lemma~\ref{l:equivalence}, $\ow D_0$
does not depend on the choice of $o\in X$. In fact,
$f\in \ow D$ belongs to $\ow D_0$ if and only if there exists a
sequence $(\varphi_n)$ in $C_c (X)$ with $\varphi_n \to f$ pointwise
and $\ow Q (\varphi_n - f)\to 0, n\to \infty$.
\end{remark}
\begin{lemma}\label{lem-properties D0}
Let $(b,c)$ be connected and $o \in X$ be fixed. Then $(\ow{D}_0, \as{\cdot,\cdot}_o)$ is a
Hilbert space. Furthermore, for each normal contraction $C:\R \to \R $
and each $f \in \ow{D}_0$ we have $C\circ f \in \ow{D}_0$.
\end{lemma}
\begin{proof}
The fact that $\ow{D}_0$ is a Hilbert space is a consequence of the lower-semicontinuity
of $\ow{Q}$ with respect to pointwise convergence. Now, let $f \in \ow{D}_0$
and a normal contraction $C$ be given. Let $(\varphi_n)$ be a sequence in $C_c(X)$
approximating $f$ with respect to $\aV{\cdot}_0$. Since $\ow{Q}(C\circ \varphi_n) \leq \ow{Q}(\varphi_n)$ by
the cut-off property, we obtain that $(C \circ \varphi_n)$ is a bounded sequence
in the Hilbert space $(\ow{D}_0, \as{\cdot,\cdot}_o)$. Thus, it has a weakly convergent
subsequence with limit $\varphi \in \ow{D}_0$. Since $\varphi_n \to f$ pointwise,
we obtain $C \circ \varphi_n \to C \circ f$ pointwise. Hence, $C\circ f  = \varphi \in \ow{D}_0$.
This finishes the proof.
\end{proof}

\medskip

Finally, we will need the concept of capacity.  In our context the
 \textit{capacity} $\mbox{cap} (x)$ of a point $x\in X$ is defined as
  $$\mbox{cap} (x) = \inf\{ \ow Q (\varphi) : \varphi\in C_c (X), \varphi (x) =1\}.$$

\medskip

We now assume that we are additionally given a measure $m$ on $X$ of
full support. Then $\ell^2 (X,m)$ is the vector space of square
summable (with respect to $m$) elements of $C (X)$. It is a
Hilbert space with respect to the inner product
$$\langle f, g\rangle :=\sum_{x\in X} f(x) g(x) m(x).$$
The associated norm is given by
$$\|f\|:=\langle f,f\rangle^{1/2}.$$
Whenever $(b,c)$ is a graph over $X$ and a measure $m$ of full
support is given, we  obtain  the bilinear  form $Q^{(D)}$ by
restricting $\ow Q$ to
\begin{align*}
 D(Q^{(D)}):=\overline{C_{c}(X)}^{\aV{\cdot}_{\ow Q}},
\end{align*}
where the closure is taken with respect to the norm
\begin{align*}
\aV{u}_{\ow Q}:={(\ow Q(u)+\aV{u}^{2})}^{1/2}.
\end{align*}
By definition  $Q^{(D)} (f,g)  = \ow Q (f,g)$ holds for $f,g \in D
(Q^{(D)})$. Thus, the key ingredient in the definition of $Q^{(D)}$
is the domain $D(Q^{(D)})$. Here, we have the following
characterization. We have not been able to find it in the
literature. Thus, we include proof in Appendix \ref{Domain}. It may be of
interest   in other situations as well.

\begin{lemma}[Characterization of
$D(Q^{(D)})$]\label{Characterization-D-Q-D} Let $(b,c)$ be a graph
over $X$ and $m$ be a measure on $X$ of full
support.  Then, $$\overline{C_{c}(X)}^{\aV{\cdot}_{\ow Q}} = \ow D_0
\cap \ell^2 (X,m).$$
\end{lemma}

There then exists a unique selfadjoint operator  $L:=L^{(D)}_m$ with
$$ \langle L f, g\rangle = Q^{(D)} (f,g)$$
for all $f$ in the domain of the operator $D(L)$  and $g\in
D(Q^{(D)})$.

This operator is non-negative and gives rise to a semigroup
$e^{-tL^{(D)}_m}, t\geq 0$, and resolvents $(L^{(D)}_m +
\alpha)^{-1}$, $\alpha >0$. The semigroup and the resolvents are
bounded operators on $\ell^2 (X,m)$. It turns out that their
restrictions to $C_c (X)$ can  be  uniquely extended to give bounded
operators on $\ell^p (X,m)$ for all $p\in [1,\infty)$, see \cite{KL1}.

\section{Uniformly transient graphs}
In this section we introduce the  class of graphs under
considerations. They can be characterized in three different ways
viz  via a transience property, via an isoperimetric inequality and
via a Feller-type property.

\bigskip

\begin{theorem}[The main characterization] \label{main} Let $(b,c)$ be a connected graph over $X$.
 Then, the following assertions are equivalent:
\begin{itemize}
\item[(i)] The inclusion $\ow D_0\subseteq C_0 (X)$ holds. (``Uniform transience'')

\item[(ii)] There exists  $C\geq 0$ with $\|\varphi\|_\infty \leq C
\ow Q^{1/2} (\varphi) $ for all $\varphi \in C_c (X)$. (``Supnorm
isoperimetricity'')

\item[(ii$^{\prime}$)]
For one (all)  $o\in X$ there exists $C_o\geq 0$ with
$\|\varphi\|_\infty \leq C_o \aV{\varphi}_o$ for all $\varphi \in
C_c (X)$.

\item[(iii)]  The inclusion $D(Q_m^{(D)})\subseteq C_0 (X)$ holds for any measure $m$ on $X$  of full support. (``Uniform strong Feller property'')
\item[(iii$^{\prime}$)] The inclusion  $D(Q_m^{(D)})\subseteq C_0 (X)$ holds for any measure $m$ on $X$  of full support with $m(X) <\infty$.

\item[(iv)] The inequality $\inf_{x\in X} \textup{cap}(x) >0$ holds. (``Uniform positive capacity of points'')
\end{itemize}
\end{theorem}

\begin{remark} Of course, one can replace the condition $\varphi \in
C_c (X)$ by $\varphi \in \ow D_0$ in (ii) and (ii$^{\prime}$).
\end{remark}

\begin{proof}
We first show (i) $\Longrightarrow$ (ii). Choose $o\in X$ arbitrary.
By (i) and the closed graph theorem the map
$$(\ow D_0, \aV{\cdot}_o) \longrightarrow (C_0
(X),\aV{\cdot}_\infty), \; f\mapsto f,$$ is continuous. Thus, there
exists a  $C_1\geq 0 $ with
$$\|f\|_\infty \leq C_1 \aV{f}_o$$
for all $f\in \ow D_0$. Therefore, it suffices  to show that there
exists a $C_2\geq 0$ with
$$\aV {\varphi}_o \leq  C_2 \ow Q^{1/2} (\varphi)$$
for all $\varphi \in C_c(X)$.

Assume the contrary. Then, we can chose a
sequence $(\varphi_n) \in C_c(X)$ with
$$ \aV{\varphi_n}_o > n \ow Q^{1/2}(\varphi_n)$$
for all $n$. Without loss of generality, we can assume that
$\aV{\varphi_n}_o = 1$ for all $n$. This yields $\ow Q
(\varphi_n)\to 0, n\to \infty$ and then $|\varphi_n (o)|\to 1, n\to
\infty$. By Lemma~\ref{l:DLip} and  $\ow Q(|\varphi_n|) \leq \ow
Q(\varphi_n)\to 0 $ it follows that $|\varphi_n| \to 1$ pointwise as
$n \to \infty$.  By Fatou's Lemma, $\ow Q(1) \leq \lim_{n \to
\infty} \ow Q(|\varphi_n|) = 0$ so that $1 \in \ow D$ and $c \equiv
0$. Then the preceding considerations show, in fact, that the
sequence $(|\varphi_n|)$ from $C_c(X)$ converges to $1$ in the sense
of $\aV{\cdot}_o$. This in turn implies $1\in \ow D_0$ which
contradicts (i).

Due to the equivalence of all norms $\aV{\cdot}_o$, $o\in X$, given
in Lemma \ref{l:equivalence} the validity of   (ii$^{\prime}$) for one $o\in X$
is equivalent to the validity of (ii$^{\prime}$) for all $o\in X$.
The implications (ii) $\Longrightarrow$ (ii$^{\prime}) \Longrightarrow$ (i) are then clear.

The equivalence between (i) and (iii) and (iii$^{\prime}$)  follows easily
from the characterization
$$D (Q_m^{(D)}) = \ow D_0 \cap \ell^2 (X,m)$$
given in Lemma \ref{Characterization-D-Q-D}.

Finally, the equivalence between (ii) and (iv) follows easily from the definition of the capacity of a point.
\end{proof}

\begin{remark}
  Note that in the above proof of (i) $\Longrightarrow$ (ii) we have actually shown that if $\ow Q^{1/2}$ and $\aV{\cdot}_o$ are not equivalent norms on $C_c(X)$, then $c \equiv 0$ and $1 \in \ow D_0$ which is equivalent to recurrence as discussed in Appendix \ref{Transience}.

\end{remark}

\begin{definition}[Uniformly transient graphs]  Let $(b,c)$ be a connected graph over $X$.
Then, $(b,c)$ is called \textit{uniformly transient} if it
satisfies one of the equivalent conditions of the previous theorem.
\end{definition}

\begin{remark}[Context of the definition]
\begin{itemize}
\item   As is well-known, see e.g. \cite{Soa} or  Appendix \ref{Transience},  a connected graph
$(b,c)$ is \textit{recurrent} if and only if  the constant function
$1$ belongs to $\ow D_0$ and $\ow Q (1) =0$ holds. It is
\textit{transient} if it is not recurrent. Obviously, $1$ cannot
belong to $\ow D_0$ if $\ow D_0$ is contained in $C_0 (X)$. Thus,
condition (ii) of Theorem \ref{main} gives that uniformly transient
graphs do indeed satisfy a very uniform version of {transience}.

\item Condition (iv) is   the definition of uniform
transience given in \cite{BCG}.

\item The semigroup $e^{-t L^{(D)}_m} $, $t \geq 0$, associated to a
graph $(b,c)$ over $X$ satisfies the \textit{Feller property} if it
maps $C_c (X)$ into $C_0 (X)$. Now, the spectral calculus easily gives
that the semigroup always  maps $\ell^2 (X,m)$ into
$D(L^{(D)}_m)\subseteq D (Q^{(D)}_{m})$ for any $t>0$. Thus, condition
(iii) and (iii$^{\prime}$) of Theorem \ref{main} give a strong form
of the Feller property. For a recent study of the Feller property on
graphs we refer the reader to \cite{Woj4}.

\item Let us emphasize that  uniform transience (like transience) does not
depend on the measure but only on the form $\ow Q$, i.e.,
the graph structure  $(b,c)$.

\item As is well-known, transience is stable under extending graphs,
i.e., transience of a subgraph implies transience of the whole graph.
This stability is not true for uniform transience. Indeed, gluing
together a uniformly transient graph with a recurrent graph will result in a graph
which is not uniformly transient.  The same is true regarding the stability of the
Feller property, see \cite{Woj4}.

\item For a probabilistic approach to transience and various further
aspects of random walks on graphs we refer the reader to the
standard monograph \cite{Woe}.

\end{itemize}

\end{remark}

We next present three  classes of graphs which are uniformly
transient.

\medskip

Recall that a connected graph is \textit{canonically compactifiable}
in the sense of \cite{GHKLW} if any function in $\ow D$ is bounded.

\begin{corollary}[Canonically compactifiable graphs are
uniformly transient] \label{cc-implies-ut} Let $X$ be an infinite set and $(b,c)$ be a connected canonically
compactifiable graph over $X$. Then, $(b,c)$ is uniformly transient.
\end{corollary}
\begin{proof} If a graph is canonically compactifiable, it is not hard to infer from the closed graph theorem
that the map $(\ow D, \aV{\cdot}_o) \longrightarrow \ell^\infty(X)$ is continuous so that
such graphs satisfy property (ii$^{\prime}$) of
Theorem \ref{main} (see \cite{GHKLW} as well for details).
\end{proof}

Let us now turn to the another large class of uniformly transient graphs. 
For a measure $m$, we say the operator $L_{m}^{(D)}$ has a \emph{spectral gap}
if the bottom of the spectrum of $L_{m}^{(D)}$ is positive.

\begin{corollary}[Spectral gap] \label{cc-strong-isoperimitry} Let $(b,c)$ a  graph
 over $X$ and suppose $L_{m}^{(D)}$ has a spectral gap for $m$
 satisfying $\delta := \inf_{x\in X} m(x) >0$ (e.g. $m\equiv 1$).
Then, $(b,c)$ is uniformly transient.
\end{corollary}
\begin{proof} As $L_{m}^{(D)}$ has a spectral gap $\lm>0$, we have $Q^{(D)}_{m}(\ph)\geq\lm\|\ph\|^{2}$
for all $\ph\in C_{c}(X)$. Thus, we have for all $\ph\in C_{c}(X)$
\begin{align*}
    \|\ph\|_{\infty}^{2}\leq \delta^{-1}  \|\ph\|^{2}\leq \delta^{-1}  \lm^{-1}\ow Q(\ph)
\end{align*}
which yields the statement by (ii) of Theorem \ref{main}. \end{proof}



The corollary above implies  that all graphs with standard weights,
i.e., $b:X\times X\to\{0,1\}$ and $c\equiv0$, which satisfy a strong
isoperimetric inequality are uniformly transient. This includes, for
example, trees with all vertex degrees at least three  and Cayley
graphs of hyperbolic groups.

In the example below we discuss that the reverse implication of
Corollary~\ref{cc-implies-ut} does not hold.
Thus,  there exist
uniformly transient graph which are not canonically compactifiable.

\begin{example}
Consider a tree with standard weights with vertex degree larger than
two. Then, the Laplacian on the tree has a spectral gap and is
uniformly transient by Corollary~\ref{cc-strong-isoperimitry}.
However, it can be seen to be not canonically compactifiable.
Consider a path of vertices $(x_{n})$ in the tree  and denote by
$T_{n}$ the subtrees emanating from $x_{n}$, $n\ge0$,  (i.e. the
vertices of $T_{n}$ are those vertices of $X$ which are closer to
$x_n$ than to $x_k$ for any $k\neq n$).  Define a function $\ph$ by
letting $\ph(x)=\sum_{j=1}^{n}j^{-1}$ for $x\in T_{n}$. It is
immediate that $\ph\in \ow{D}$ but $\ph$ is not bounded.

Another, more abstract way, to see that the tree is non-compactifiable follows
from \cite{GHKLW} (using the notation of \cite{GHKLW}):
by \cite[Theorem~4.3.]{GHKLW} a graph is canonically
compactifiable if and only if the diameter with respect to the
metric $\rho$, which is the square root of the free effective
resistance, is finite. On trees $\rho^{2}$ equals the metric $d$
\cite[Lemma~8.1.]{GHKLW} which is the path metric with weights
$1/b(x,y)$. Clearly, $d$ has infinite diameter in the standard
weight case.
\end{example}

Finally, recall that a graph $(b,c)$ over $X$ is called
\textit{quasi-vertex-transitive} if there exists an $n\in\N$ and
vertices $x_1,\ldots, x_n$ such that for any  vertex $y$ in $X$
there exists  an $j\in\{1,\ldots, n\}$ and a  bijection $h :
X\longrightarrow X$ with $h(y) = x_j$ and  $c (h(z)) = c (z)$ and $
b (h(v), h(w)) = b(v,w)$ for all $z,v,w\in X$. If $n$ can be chosen
as $1$, the graph is called \textit{vertex-transitive}.

\begin{corollary}\label{c:vertrextransitive}
Whenever a  graph is both  quasi-vertex-transitive and transient it
is also uniformly transient.
\end{corollary}
\begin{proof} By transience and (iii) of
Theorem \ref{char-transience}  for any vertex $o$ in the
graph there exists a constant $C_o$ with $|\varphi(o)|^2 \leq C_o \ow Q
(\varphi)$ for all $\varphi \in C_c (X)$. By
quasi-vertex-transitivity the constants $C_o$ can be chosen
independently of $o\in X$. Now, the desired statement follows directly
from Theorem \ref{main}.
\end{proof}

\section{A metric criterion for uniform transience}
In this section we present a   characterization  for uniform
transience in terms of boundedness with respect to a certain metric.

\bigskip

Let $(b,c)$ a connected graph over the countably infinite  $X$ and $o\in X$ be arbitrary.  We define for $x,y\in X$
\begin{eqnarray*} \gamma_o (x,y) & :=&\sup\{ |\varphi(x) - \varphi (y) | : \varphi
\in C_c (X), \aV{\varphi}_o\leq 1\}\\
& = &\sup\{ |f(x) - f(y)| : f\in \ow D_0, \aV{f}_o \leq 1\}.
\end{eqnarray*}

Here, the last equality follows by approximation and Lemma
\ref{l:pointevaluation}. Similarly, we  define
 for $x,y\in X$
\begin{eqnarray*}
\gamma (x,y) & := &\sup\{ |\varphi(x) - \varphi (y) | : \varphi \in
C_c (X), \ow Q (\varphi) \leq 1\} \\
&= &\sup\{ |f(x) - f(y)| : f\in \ow D_0, \ow Q (f)  \leq 1\}.
\end{eqnarray*}

The crucial properties of $\gamma_o$ and $\gamma$ are given in the
next proposition.

\begin{proposition}[Properties of $\gamma$ and $\gamma_o$]
The map $\gamma_o: X\times X\longrightarrow [0,\infty)$ is a metric
for any $o\in X$ and so is the map $\gamma : X\times
X\longrightarrow [0,\infty)$. Any $f\in \ow D_0$ is uniformly
continuous with respect to $\gamma$. More specifically, $f\in \ow D_0$ satisfies
$$ |f (x) - f(y)| \leq  \ow Q (f)  \gamma (x,y)$$
for all $x,y\in X$. Moreover,  $\gamma_o \leq \gamma$ and
$$\gamma = \sup_{o\in X} \gamma_o$$ holds.
\end{proposition}

\begin{proof} We first show that $\gamma_o$ is a metric:  The values of $\gamma_o$ are finite
by Lemma~\ref{l:DLip}. Symmetry and triangle inequality are clear. As the
characteristic function of any $x\in X$ belongs to $C_c (X)$, the map $\gamma_o$ is not degenerate.
Similarly, it can be shown that $\gamma$ is a metric.

\smallskip

The  statement concerning uniform continuity  is clear from the
definition of $\gamma$.

\smallskip

From the definition of $\gamma$ and $\gamma_o$ is is clear that
$\gamma_o\leq \gamma$ holds for any $o\in X$. To show the statement
on the supremum, let $x,y\in X$ be given  and choose  for
$\varepsilon
>0$ arbitrary $\varphi \in C_c (X)$ with $\ow Q (\varphi) \leq 1$
and
$$\gamma (x,y) \leq |\varphi (x) - \varphi(y)|  + \varepsilon.$$
If we now choose $o\in X$ with $\varphi (o) =0$, then we obtain
$\aV{\varphi}_o = \ow Q (\varphi)^{1/2} \leq 1$ and, hence,
$\gamma_o (x,y) \geq |\varphi (x) - \varphi (y)|$. Altogether, we
arrive at
$$\gamma (x,y) \leq \gamma_o (x,y) + \varepsilon.$$
As $\varepsilon >0$ is arbitrary this gives the desired statement on
the supremum.
\end{proof}

\begin{remark}
\begin{itemize}
\item  It is not hard to see that the supremum over $f\in \ow D_0$ can be replaced by a
maximum both for $\gamma$ and $\gamma_o$ (compare \cite{GHKLW} for a
similar reasoning).

\item If $(b,c)$ is transient, then $\aV{\cdot}_o$ and $\ow Q^{1/2}$
are equivalent norms on $C_c (X)$  for any $o\in X$ (see Appendix
\ref{Transience}). Thus, in this case, $\gamma$ and $\gamma_o$ are
equivalent metrics (see Corollary \ref{gamma-eq-gamma-o}).

\item It seems that the  metric in the definition is the square root of the metric
denoted  as \textit{wired resistance metric} in \cite{JP2} (in the
transient case).

\item The analogous situation where the supremum is taken over $f \in \ow D$ instead of
$f \in \ow D_0$ has received quite some attention (see e.g.
\cite{GHKLW} and references therein). The arising metric is the
(square root of the) \textit{resistance metric}. It has also played
a role in considerations inspired by non-commutative geometry
\cite{Davies2,HKT}.
\end{itemize}
\end{remark}

Recall that a metric space is said to have \textit{finite diameter}
if there exists  $C\geq 0$ such that the distance between any two
points is bounded by $C$.

\begin{theorem}[Metric criterion for uniform transience]
\label{metric-criterion} Let $(b,c)$ a connected graph over $X$. Then, the following statements are
equivalent:
\begin{itemize}
\item[(i)] The graph $(b,c)$ is uniformly transient.
\item[(ii)] The diameter of $(X,\gamma_o)$ is finite for one (all)
$o\in X$.
\item[(iii)] The diameter of $(X,\gamma)$ is finite.
\end{itemize}
\end{theorem}
\begin{proof} (i) $\Longrightarrow$ (iii): By uniform transience and (i) of Theorem
\ref{main} there exists $C\geq 0$ with $\|\varphi\|_\infty \leq C
Q (\varphi)^{1/2}$ for all $\varphi \in C_c (X)$. This directly shows
$$ |\varphi (x) - \varphi (y)|\leq 2 \|\varphi\|_\infty \leq 2 C$$
for all $x,y\in X$ and all $\varphi \in C_c (X)$ with $\ow Q
(\varphi) \leq 1$. This implies  (iii).

\smallskip

(iii) $\Longrightarrow$ (ii): By the previous proposition, we have
$\gamma_o \leq \gamma$ for any $o\in X$. This gives (ii) (for all
$o\in X$).

\smallskip

(ii) $\Longrightarrow$ (i): Let  (ii) be valid for one $o\in X$.
Note that by Lemma \ref{l:equivalence} it then follows that (ii) is valid for all $o \in X$.
There then exists $C\geq 0$ such that $|f(x) - f(o)|\leq C$ for any
$x\in X$ and any $f\in \ow D_o$ with $\aV{f}_o \leq 1$. This gives
$$ |f(x)| \leq |f(x) - f(o)| + |f(o)| \leq C + \aV {f}_o \leq C + 1$$
for any $x\in X$ and any  $f\in \ow D_o$ with $\aV{f}_o\leq 1$. This
then implies
$$\|f\|_\infty \leq (C + 1) \aV{f}_o$$
for any $f\in \ow D_0$ and by part (ii$^{\prime}$) of Theorem \ref{main} the
desired statement  follows.


\end{proof}

\section{Uniform transience, the Royden compactification and the Dirichlet problem on the
boundary} \label{Royden} In this section we first discuss a
characterization of uniform transience in terms of the Royden
boundary of a graph. This will then allow us to show unique
solvability of the Dirichlet problem for uniformly transient graphs.

\medskip

Recall that the Royden compactification of a graph $(b,c)$ is the
unique (up to homeomorphism) compact Hausdorff space $R$
such that the following three conditions are satisfied:
\begin{itemize}
 \item $X$ is a dense open subset of $R$.
 \item Each function of the algebra $\ow{D}\cap \ell^\infty(X)$ can be uniquely extended to a continuous function on $R$.
 \item The algebra $\ow{D}\cap \ell^\infty(X)$ separates the points of $R$.
\end{itemize}

One can construct $R$ by applying Gelfand theory to the algebra generated by the uniform
closure of $\ow{D}\cap \ell^\infty(X)$ and the constant function $1$ which is a commutative
 $C^*$-algebra. For more details of this construction for graphs, we refer the reader to Section 4 of
 \cite{GHKLW} (see \cite{Roy} for the original work of Royden on
 manifolds).

\begin{definition}[Royden algebra $\mathcal{A}$]
Let $(b,c)$ be a connected graph over $X$. The uniform closure of  $ \ow{D} \cap \ell^\infty(X)$ in $\ell^\infty(X)$
is called the {\em Royden algebra} of $(b,c)$ and is denoted by $\A$. The unique extension of
$f \in \A$ to a function on $R$ will be denoted by $\hat{f}$.
\end{definition}

\begin{remark}
 \begin{itemize}
 \item Since $\ow{D} \cap \ell^\infty(X)$ separates the points of $R,$ the algebra $\A + {\rm Lin}\{1\}$
 is isomorphic to $C(R)$ by the Stone-Weierstrass theorem.
   \item It was shown in \cite{GHKLW} that $1 \in \A$ if and only if $c \in \ell^1(X)$.
  \item  For a different construction of $R$ (when $c\equiv 0$)  using a somewhat smaller
  Banach algebra  and further discussion we refer the reader to Chapter~6 of \cite{Soa}.
 \end{itemize}
\end{remark}

The set $\partial_R X = R \setminus X$ is called the {\em Royden
boundary} of $(b,c).$ The importance of the Royden boundary is due to the fact that
harmonic functions in $\ow{D}\cap\ell^\infty(X)$ are uniquely determined by their values on the closed subset
$$\partial_h X := \{x \in \rR\, : \, \hat{f}(x) = 0 \text{ for all } f\in \ow{D}_0 \cap \ell^\infty(X) \},$$
see the discussion below. We call $\rH$ the {\em harmonic boundary} of $(b,c)$.  In general it is strictly smaller
than the Royden boundary. However, it turns out that the validity of
$\partial_R X = \rH$ is equivalent to uniform transience.

\begin{theorem} \label{thm-char uniform transience royden boundary}
 Let $(b,c)$ be a connected graph over $X$. Then, the following assertions are equivalent:
 \begin{itemize}
  \item[(i)] $(b,c)$ is uniformly transient.
  \item[(ii)] The equality $\rH = \rR$ holds.
 \end{itemize}
\end{theorem}
\begin{proof}
(i) $\Longrightarrow$ (ii): Assume $(b,c)$ is uniformly transient, i.e.,
$\ow{D}_0 \subseteq C_0(X)$ holds. Since $X$ is dense in $R$, we
can approximate any $x \in R$ by a net $(x_i) \subseteq X$. Any
such net converging to a boundary point will eventually leave every
finite subset of $X$. As functions in $C_0(X)$ eventually become
arbitrarily small, we infer $\lim_if(x_i) =0$, for each $f \in
\ow{D}_0 \subseteq C_0(X)$. Now, the statement follows from the
continuity of $\hat{f}$ on $R$ and the fact that $\hat{f}|_X = f$.

(ii) $\Longrightarrow$ (i): Assume $\rH = \rR$ and suppose $(b,c)$ is
not uniformly transient. Then there exists a function $f \in
\ow{D}_0$, a constant $C > 0$ and a sequence $(x_n) \subseteq X$
leaving every finite subset of $X$ such that
$$|f(x_n)| \geq C, \text{ for all } n\geq 1.$$
Without loss of generality we may assume that $f$ is  bounded. As
$R$ is compact, the sequence $(x_n)$ has a convergent subnet with
limit $x \in \rR$. From the continuity of $\hat{f}$ we infer
$|\hat{f}(x)| \geq C > 0$. But this implies $x \not\in \rH,$ which
is a contradiction.
\end{proof}

\medskip

We will now study the relation of  $\rH$ and harmonic functions in $\A$. Let us first recall
the definition and some properties of harmonic functions. Given a weighted graph $(b,c)$
over $X$ we introduce the associated {\em formal Laplacian} $\LL$ acting on
$$\ow{F}:= \{ f \in C(X) : \sum_{y \in X} b(x,y) |f(y)| < \infty \text{  for all } x\in X \}$$
as
$$\LL f(x) := \sum_{y \in X} b(x,y) (f(x) - f(y)) + c(x) f(x).$$
The operator $\LL$ can be seen as a discrete analogue to the Laplace-Beltrami operator
on a Riemannian manifold. We will be interested in {\em harmonic functions,} i.e., functions
 $f \in \ow{F}$ satisfying $\LL f = 0$. The formal operator $\LL$ is intimately linked to the
  form $\ow{Q}$ by the following lemma.

\begin{lemma}[Green's formula] \label{lem-green formula}
 The inclusion $\ow{D} \subseteq \ow{F}$ holds and for each $f \in \ow{D}$ and each $g \in C_c(X)$ the equality
 $$\ow{Q}(f,g) = \sum_{x \in X} (\LL f)(x) g(x)$$
 is satisfied. Furthermore, if $f\in \ow{D}$ is harmonic, the above equality extends to all $g \in \ow{D}_0$ and is equal to $0$.
\end{lemma}

\begin{proof}
We first show the inclusion $\ow{D} \subseteq \ow{F}$ by using the argument of the proof of
Proposition 3.8 of \cite{HKLW}. Letting $B_x := \sum_{y \in X} b(x,y)$ we estimate
\begin{align*}
 \sum_{y \in X} b(x,y)|f(y)| &\leq  \sum_{y \in X} b(x,y)|f(x) - f(y)| + \sum_{y \in X} b(x,y)|f(x)| \\
 & \leq \Big(\sum_{y \in X} b(x,y) \Big)^{\frac{1}{2}} \Big( \sum_{y \in X} b(x,y)|f(x) - f(y)|^2\Big)^{\frac{1}{2}} + B_x |f(x)|\\
 &\leq B_x^{1/2} \ow{Q}^{1/2}(f) + B_x |f(x)|
\end{align*}
which shows the claim. Combining $\ow{D} \subseteq \ow{F}$ and Lemma 4.7 of \cite{HK} we obtain the equality
 $$\sum_{x \in X} \LL f(x) g(x) = \ow{Q}(f,g) $$
for $f \in \ow{D}$ and $g \in C_c(X).$ The furthermore statement follows from this equality
and the denseness of $C_c(X)$ in $\ow{D}_0$ with respect to $\aV{\cdot}_o$.
\end{proof}

The following lemmas are well-known and easy to check, confer Section~5 of \cite{GHKLW}.

\begin{lemma} \label{lem-superharmonic}
 Assume that $(b,c)$ is connected. If $f \in \ow{F}$ is non-negative and not constant and
 satisfies $\LL f\leq 0$ on $X$, then $f$ does not attain a maximum on $X$.
\end{lemma}

\begin{lemma}\label{lem-harmonic}
 If $f \in \ow{F}$ satisfies $\LL f = 0$ on $X,$ then $|f|$ satisfies $\LL|f| \leq 0.$
\end{lemma}

\begin{corollary}[Maximum principle for uniformly transient graphs] \label{cor-maxprinciple-ut}
Assume $(b,c)$ is uniformly transient and let $f \in \A$ be harmonic. Then,
$$\|f\|_\infty = \|\hat{f}|_{\rH}\|_\infty.$$
\end{corollary}

\begin{proof}
Combining Lemma \ref{lem-superharmonic} and Lemma \ref{lem-harmonic} we conclude
the a harmonic function $f\in \A$ does not attain its maximum on $X$.  As its continuation
$\hat f$ is continuous on the compact set $R,$ it attains its maximum at $\rR$. Since
uniform transience implies $\rH = \rR$, the statement follows.
\end{proof}

\begin{remark}
 The maximum principle shows that on uniformly transient graphs harmonic functions
 in the Royden algebra $\A$ are uniquely determined by their value on the harmonic boundary.
  In  the next section we will prove that an analogous statement holds for general transient
  graphs and harmonic functions in $\ow{D} \cap \ell^\infty(X)$ (instead of $\A$).
\end{remark}

For our subsequent considerations we will need the following
well-known statement, see e.g. Theorem 6.3 in \cite{Soa}.

\begin{proposition}[Royden decomposition for uniformly transient graphs]\label{proposition-decomposition-ut}
Let $(b,c)$ be a connected uniformly transient graph over $X$. Then for any $f\in \ow D$ there
 exist unique $f_0  \in \ow D_0$ and $f_h
\in \ow D$ harmonic with $f = f_0 + f_h$. Moreover, if $f \in \ell^\infty(X)$, then $f_h \in \ell^\infty(X)$.
\end{proposition}
\begin{proof}

{\em Uniqueness:} Assume there exist $g_0,f_0 \in \ow{D}_0$ and harmonic functions
 $g_h,f_h \in \ow{D}$ such that $f = f_0 + f_h = g_0 + g_h$. Since  $g_h - f_h$ is harmonic
  and $f_0 - g_0 \in \ow{D}_0,$ by Lemma \ref{lem-green formula} we obtain
$$0 = \ow{Q}(g_h - f_h, f_0 - g_0) = \ow{Q}(f_0 - g_0).$$
Thus, the connectedness of $(b,c)$ implies that $f_0 - g_0$ is constant. If $c \neq 0$,
then we obtain $f_0 - g_0 = 0$ immediately.  If $c \equiv 0$, then the transience of $(b,0)$
  implies that the only constant function in $\ow{D}_0$ vanishes everywhere
   (see e.g. Theorem~\ref{char-transience}). This shows uniqueness.

{\em Existence:} By standard Hilbert space theory, there exists a minimizer of the
functional $u \mapsto \ow{Q}(u-f)$ on the set $\ow{D}_0$ which we denote by
$f_0$. Now, let $\varphi \in C_c(X)$  and $\varepsilon > 0$ be arbitrary. Then,
since $f_0 + \varepsilon \varphi \in \ow{D}_0$, we obtain
$$ \ow{Q}(f_0-f) \leq \ow{Q}(f_0 + \varepsilon \varphi -f  ) = \ow{Q}(f_0-f) + 2\varepsilon \ow{Q}(f_0 -f,\varphi) + \varepsilon^2 \ow{Q}(\varphi).$$
As $\varepsilon$ and $\varphi$ were arbitrary this shows that $f_h:= f - f_0$ is harmonic.

{\em The ``moreover'' statement:} Assume that $f$ is bounded. Since $(b,c)$ is
 uniformly transient, we have $f_0 \in C_0(X) \subseteq \ell^\infty(X)$ and the
 statement follows from $f_h = f - f_0$. This finishes the proof.
\end{proof}

We will now show that for each function  $\varphi $ in
$$ C_0(\rH) := \{\hat{f}|_\rH :  f \in \A \} $$ the Dirichlet problem
$$\LL f = 0 \text{ on }X \quad \hat f|_\rH = \varphi$$
has a unique solution provided $(b,c)$ is uniformly transient. Let us first identify the space $C_0(\rH)$.

\begin{lemma}
 Let $(b,c)$ be transient. If we equip $\rH$ with the subspace topology and
 denote by $C(\rH)$ its continuous functions the following is true:
 \begin{itemize}
  \item  If $1 \in \A$, the equality $C_0(\rH) = C(\rH)$ holds.
  \item If $1 \not \in \A$, there exists a point $\infty \in \rH$ such that
  $$C_0(\rH) = \{f \in C(\rH) :  f(\infty) = 0\}.$$
 \end{itemize}
\end{lemma}
\begin{proof}
 The inclusion $C_0(\rH) \subseteq C(\rH)$ is obviously satisfied.  As $\rH$ is compact
 each function in $C(\rH)$ can be extended to a function in $C(R)$ by Tietze's extension
  theorem.  If $1 \in \A$, the algebra  $\A$ is isomorphic to $C(R)$ and the equality $C_0(\rH) = C(\rH)$
  follows. If $1 \not \in \A$, the functions in $\A$ vanish at exactly one point $\infty \in \rR$
  (as otherwise, since $\A$ separates  points, the Stone-Weierstrass theorem would imply
  $1 \in \A$). Since $\ow{D}_0 \subseteq \A$, we obtain $\infty \in \rH$. This finishes the proof.
\end{proof}

\begin{theorem}[The DP on uniformly transient graphs] \label{thm-DP ut graphs}
 Assume $(b,c)$ is connected and uniformly transient. For each $\varphi \in C_0(\rH)$ the equation
 $$\LL f = 0 \text{ on }X \quad \hat{f}|_{\rH} = \varphi$$
 has a unique solution $f_\varphi \in \A$. Furthermore, the mapping $C_0(\rH) \to \A,$ $\varphi \mapsto f_\varphi$ is an isometry.
\end{theorem}

\begin{proof}
We will only treat the case where $1 \in \A$. The other case can be treated similarly.

{\em Uniqueness and the isometry property:} This is an immediate consequence of the
maximum principle for uniformly transient graphs, Corollary \ref{cor-maxprinciple-ut}.

{\em Existence: } Consider the set  $G \subseteq C_0(\rH)$ of functions $\varphi$ for
which there exists $f \in \ow{D} \cap \ell^\infty(X)$ such that $\hat{f} = \varphi$ on $\rH$.

 {\em Step 1:} For each $\varphi \in G$ there exists a solution to the boundary value problem.

  {\em Proof of Step 1:}  Assume that $\varphi = \hat f$ with $f \in \ow{D} \cap \ell^\infty(X)$.
  By the Royden decomposition, Proposition~\ref{proposition-decomposition-ut}, the function
  $f$ has a unique decomposition $f = f_0 + f_h$ with $f_0  \in \ow{D}_0$ and $f_h \in \ow{D}\cap \ell^\infty(X)$
   harmonic.  Since $\ow{D}_0$ vanishes on $\rH$, we obtain $\varphi = \hat{f} = \hat{f}_h$ on $\rH$
   and the claim follows.

  {\em Step 2:} Suppose that $f_n \in \ow{D}\cap \ell^\infty (X)$ solves $\LL f_n = 0$  and
  $\hat f _n |_{\rH}= \varphi_n$. Furthermore, suppose $\varphi_n \to \varphi$ uniformly as
   $n \to \infty.$ Then, $(f_n)$ converges uniformly to $f \in \A$ that solves $\LL f = 0$ and
    $\hat f |_{\rH}= \varphi.$

 {\em Proof of Step 2:} By the maximum principle, Corollary \ref{cor-maxprinciple-ut}, we have
 $$\|f_n - f_m \|_\infty  = \|\varphi_n - \varphi_m\|_\infty \to 0$$
 which implies the uniform convergence of $(f_n)$ to some $f \in \ell^\infty(X)$. We obtain
 $f \in \A$ by the definition of $\A$ and $\hat f = \varphi$ on $\rH$ by the uniform convergence.
  It follows from Lebesgue's theorem of dominated convergence that $\LL f_n \to \LL f.$
  This finishes the proof of Step 2.

 We can now conclude the existence part as follows: Since $\ow{D} \cap \ell^\infty(X)$
 separates the points of $R$, the set $G$ separates the points of $\rH$. Furthermore,
 as $1 \in \A$, it vanishes nowhere. Thus, by the Stone-Weierstrass theorem, $G$ is
  dense in $C(\rH) = C_0(\rH)$ and the statement follows by combining Step~1 and Step~2.
\end{proof}

\begin{remark}
Let us put the above theorem into the perspective of the existing literature.  Even though the
existence of solutions to the Dirichlet problem is well known, see e.g.  Theorem 6.47 in \cite{Soa},
uniqueness statements for the Dirichlet problem for large classes of graphs seem to be rather new,
confer \cite{GHKLW} for the corresponding result for canonically compactifiable graphs.
We would also like to emphasize the simplicity of our arguments compared to the
discussion in \cite{Soa} which is based on harmonic measures on $\rH$. Combining
the Royden decomposition (based on Hilbert space arguments) and a maximum principle
(based on the compactness of $X \cup \rH$ in the uniformly transient setting) together with
the Stone-Weierstrass theorem already yields existence. In the next section we will also
demonstrate how to use this method  for arbitrary transient graphs. However, one needs
to exercise some more care in this case when proving a maximum principle as $\rH$ might
be strictly smaller than $\rR$.
\end{remark}

\section{The Dirichlet problem on general
graphs}\label{DP-general-graphs}

In this section we show how to solve the Dirichlet problem for arbitrary connected
graphs $(b,c)$ over $X$ by only using a maximum principle and the Royden decomposition.

\begin{proposition}[Royden decomposition]\label{proposition-decomposition}
Let $(b,c)$ be a connected transient graph over $X$. Then,
to any $f\in \ow D$, there exist unique $f_0  \in \ow D_0$ and $f_h
\in \ow D$ harmonic with $f = f_0 + f_h$. Moreover, if $-a \leq f \leq b$ for
some $a,b \geq 0$ then $-a\leq f_h \leq b$. In particular, if $f \in \ell^\infty(X)$, then $f_h \in \ell^\infty(X)$.
\end{proposition}
\begin{proof}

Existence and uniqueness of the decomposition can be proven as in the
proof of the Royden decomposition, Proposition~\ref{proposition-decomposition-ut}.

{\em The ``moreover'' statement:} Suppose $-a \leq f \leq b$ for some $a,b \in \R$.
Since $C_c(X)$ is dense in $\ow{D}_0$ with respect to $\ow{Q}$ and by the construction
of $f_h$ (see proof of the Royden decomposition, Proposition~\ref{proposition-decomposition-ut}),
there exists a sequence $(\varphi_n) \subseteq C_c(X)$ such that
$$\ow{Q}(\varphi_n - f) \to \ow{Q}(f_h), \text{ as } n \to \infty.$$
Since we assumed the bound $-a\leq f\leq b,$ the equality $((f - \varphi_n)\vee (-a))\wedge b = f - \psi_n $
holds for some compactly supported function $\psi_n$. By the cut-off property of $\ow{Q}$ we obtain
$$\ow{Q}(\varphi_n - f) \geq  \ow{Q}(\psi_n - f) = \ow{Q}(\psi_n - f_0) + \ow{Q}(f_h) \geq \ow{Q}(f_h),$$
showing the convergence $\psi_n \to f_0$ with respect to $\ow{Q}$. By transience, this implies
 pointwise convergence $\psi_n \to f_0,$  see Theorem \ref{char-transience}. As by construction t
 he inequalities $-a \leq f - \psi_n \leq b$ hold, we obtain the statement.
\end{proof}

\begin{remark}
 In \cite{Soa} the above proposition is stated without the positivity assumption on $a,b$.
  However, as we deal with possibly non vanishing potentials $c$, we need to assume $a,b\geq 0$
   to ensure the inequality $$\ow{Q}( (f \vee (-a))\wedge b) \leq \ow{Q}(f).$$
\end{remark}

The following proposition is a variant of Theorem 6.7 in \cite{Soa}.

\begin{proposition} \label{proposition-lowerbound} Assume $(b,c)$ is transient.
Let $f\in \ow{D} \cap \ell^\infty(X)$ be harmonic and assume  $\hat{f} \geq - C$
on $\partial_h X$ for some $C\geq 0$. Then, $f \geq -C$ on $X$.
\end{proposition}

\begin{proof}
   For $\varepsilon > 0$ let  $F := \{x \in R \, : \, \hat{f}(x) + \varepsilon \leq -C\}.$
   By our assumptions we have $F \cap \rH = \emptyset.$ Thus, for each $x \in F$,
   there exists a function $g_x \in \ow{D}_0 \cap \ell^\infty(X),$ such that $\hat{g}_x(x) \neq 0$.
   By Lemma \ref{lem-properties D0} we have $|g_x| \in \ow{D}_0$. Thus, we may
   assume $\hat{g}_x \geq 0$ on $R$.  Furthermore,  let $U_x$ be a neighborhood of
   $x$, such that $\hat{g}_x(y) > 0$ for each $y \in U_x$. Obviously, $F$ is closed and,
   hence, compact. Thus, there exist finitely many points $x_i \in F$  such that the corresponding $U_{x_i}$ cover $F$. With
 $$\ow{g} = \sum_i \hat{g}_{x_i}, $$
 we let $\alpha = \inf\{ \ow{g}(x)\, : \, x \in F\}$ and set $g = \min\{1, \alpha^{-1}\ow{g}\}$.
 Then, clearly, $g \geq 1$ on $F$ and by Lemma \ref{lem-properties D0} the restriction of
  $g$ to $X$ belongs  to $\ow{D}_0 \cap \ell^\infty(X)$. By choice of the set $F$ we obtain
 $$f + \|f\|_\infty g \geq  -\varepsilon - C \text{ on } X.$$
  Applying the Royden decomposition, Proposition~\ref{proposition-decomposition}
  to the function $f + \|f\|_\infty  g$ and noting that $f$ is its harmonic part, we obtain
  $f \geq  -\varepsilon - C$. Since $\varepsilon$ was arbitrary, the claim follows.
\end{proof}

\begin{corollary}[Maximum principle] \label{cor-maxprinciple}
Assume $(b,c)$ is transient and that the harmonic boundary is non-empty.
Then, for each harmonic $f \in \ow{D} \cap \ell^\infty(X)$ the equality
$$\|f\|_\infty = \|\hat{f}|_{\rH}\|_\infty$$
holds.
\end{corollary}
\begin{proof}
The statement is a direct consequence of  Proposition~\ref{proposition-lowerbound}.
\end{proof}

\begin{remark}
 Note that, in general, the maximum principle holds for harmonic functions in
 $\ow{D}\cap \ell^\infty(X)$ only. Indeed, the failure of the maximum principle for functions
 in $\A$ may lead to non uniqueness of solutions to the Dirichlet problem.
\end{remark}

It may happen that $\rH = \emptyset$ even if $(b,c)$ is transient. This is due to the fact
that graphs $(b,c)$ with non vanishing potential $c$ are always transient, see Theorem \ref{char-transience}.
Indeed, the following characterization holds.

\begin{proposition}
 Assume $(b,c)$ is connected. Then, the following assertions are equivalent.
 \begin{itemize}
  \item[(i)] $\rH = \emptyset$.
  \item[(ii)] $1 \in \ow{D}_0$.
  \item[(iii)] $c \in \ell^1(X)$ and $(b,0)$ is recurrent.
 \end{itemize}
\end{proposition}
\begin{proof}
 The implication (ii) $\Longrightarrow$ (i) is immediate from the definitions. To show the
 reverse implication (i) $\Longrightarrow$ (ii) we use the construction of $g$ in the proof
 of Proposition \ref{proposition-lowerbound} (with the set $F$ being replaced by $R$)  to obtain $1 \in \ow{D}_0$.

 We let $\ow{Q}', \ow{D}'$ and  $\ow{D}_0'$ be the form and the corresponding spaces associated to the graph $(b,0)$.

 (ii) $\Longrightarrow$ (iii): Since $\ow{Q}'(f) \leq \ow{Q}(f),$ we obviously have
 $\ow{D}_0 \subseteq \ow{D}_0'$. Thus, if $1 \in \ow{D}_0$, we obtain $1 \in \ow{D}_0'$
 and $c \in \ell^1(X)$. As $\ow{Q}'(1) = 0,$ the graph $(b,0)$ is recurrent.

 (iii) $\Longrightarrow$ (ii): Let  $(b,0)$ be recurrent and $c \in \ell^1(X)$. Then, by the definition
 of recurrence given in Appendix \ref{Transience}, there exists a sequence of compactly supported
  functions $(\varphi_n)$ converging to $1$ pointwise such that $\ow{Q}'(\varphi_n) \to 0$ as
  $n\to \infty$. By the cut-off property of $\ow{Q}'$ we may assume  $0 \leq \varphi_n \leq 1$
  for each $n$. Using Lebesgue's Theorem and $c \in \ell^1(X)$ this shows $\ow{Q}(1 - \varphi_n)  \to 0,$
  as $  n \to \infty$  and, hence, $1 \in \ow{D}_0.$
\end{proof}

\begin{theorem}[Existence of solutions to the DP] \label{thm-existence of solutions}
 Assume $(b,c)$ is connected. For each $\varphi \in C_0(\rH)$ the equation
 $$\LL f = 0 \text{ on }X \quad \hat{f}|_{\rH} = \varphi$$
 has a solution $f_\varphi \in \A$.
\end{theorem}

\begin{proof}
We only need to consider the case when $\rH \neq \emptyset$. If $(b,c)$ is recurrent, we then
have $1 \in \ow{D}_0$ and hence $\rH = \emptyset$. Thus, we may assume that $(b,c)$ is transient.
Now, the proof can be carried out as in the existence part of the proof of Theorem \ref{thm-DP ut graphs}
 using the Royden decomposition  (Proposition \ref{proposition-decomposition}) and  the maximum
 principle (Corollary \ref{cor-maxprinciple}) for transient graphs with non-vanishing harmonic boundary.
\end{proof}

\begin{remark}
\begin{itemize}
 \item The class of functions on the boundary for which we prove existence of solutions to the
  Dirichlet problem is somewhat smaller that the one in \cite{Soa}. However, using further
  monotone convergence arguments we could recover these results. We refrain from giving details.
 \item As  mentioned above, uniqueness of solutions is not clear anymore as the maximum principle
  does not seem to hold for arbitrary harmonic functions in $\A$.
\end{itemize}

\end{remark}

\section{Canonically compactifiable graphs and the one-point compactification }
In this section we will have a look at canonically compactifiable
graphs in the context of uniformly transient graph. In particular,
we present a characterization of canonical compactifiability in
terms of uniform transience. Moreover, we will  provide necessary
and sufficient conditions for the Royden compactification to agree
with the one-point compactification. Finally, we will show how
Euclidean lattices in dimension at least three serve as examples
for our results.

\bigskip

We can now characterize the canonically compactifiable graphs within
the class of uniformly transient graphs.  Recall that canonically compactifiable
means that $\ow D \subseteq \ell^\infty(X)$ and that function in $\ow D$ are
said to have finite energy.

\begin{theorem}\label{cc-within-ut} Let $(b,c)$ be a graph over $X$.
Then, the following assertions are equivalent:
\begin{itemize}
\item[(i)] The graph $(b,c)$ is canonically compactifiable.
\item[(ii)] The graph $(b,c)$ is uniformly transient and any
harmonic function of finite energy is bounded.
\end{itemize}
\end{theorem}
\begin{proof} (i) $\Longrightarrow$ (ii): By the very definition of
canonical compactifiability any  function of finite energy
is bounded. Moreover, it has already been shown in Corollary~\ref{cc-implies-ut} that any canonically compactifiable graph is
uniformly transient. Thus, (i) implies (ii).

\smallskip

(ii) $\Longrightarrow$ (i): As the graph is uniformly transient, we
have $\ow D_0 \subseteq C_0 (X)$ and any element of $\ow D_0$ is
bounded. Moreover, by assumption, any harmonic function of finite energy is
bounded. Thus, the Royden decomposition, Proposition~\ref{proposition-decomposition-ut}, shows boundedness of all
functions of finite energy and (i) follows.
\end{proof}

The next result shows that uniformly transient graphs as well as
canonically compactifiable graphs  appear naturally whenever the
Royden compactification is the one-point compactification.

\begin{theorem}\label{Liouville-one-point}
Let $(b,0)$ be a connected  graph. Then, the following assertions
are equivalent:

\begin{itemize}

\item[(i)] The graph $(b,0)$ is uniformly transient and any  harmonic function
of finite energy is constant.


\item[(ii)] The graph $(b,0)$ is canonically compactifiable and any  harmonic
function of finite energy is constant.

\item[(iii)] The graph $(b,0)$ is transient and the Royden compactification of $X$ is the one-point
compactification.
\end{itemize}
\end{theorem}
\begin{proof} The equivalence between (i) and (ii) is immediate from
the previous theorem.

\smallskip

(i) $\Longrightarrow$ (iii): By the assumption on the harmonic
functions and the Royden decomposition, Proposition~\ref{proposition-decomposition}, we obtain
that the smallest algebra $\mathcal{A}_{0}$ containing $\ow D$ and the
constant functions is given as $\mathcal{A}_{0}  = \ow D_0 + \mathrm{Lin}\{1\}$,
where $\mathrm{Lin}\{1\}$ denotes the linear span of the constant functions.
Moreover, uniform transience implies $\ow D_0 \subseteq C_0 (X)$. Now,
(iii) is immediate.

\smallskip

(iii) $\Longrightarrow $ (i): This can be inferred from \cite{Soa}.
We include a proof for the convenience of the reader. As the graph
is transient, the algebra $\mathcal{A}_0:=\ow D_0 \cap \ell^\infty
(X)$ does not contain any non vanishing  constant functions. As the
Royden compactification has only one boundary point, any element of
$\mathcal{A}_0$ must then actually vanish on the Royden boundary.
(Assume the contrary: then  $\mathcal{A}_0$ contains an element of the
form $1 + f$ with $f$ vanishing at  the boundary point. By adding a
suitable function with compact support we then obtain that
$\mathcal{A}_0$ contains a uniformly positive function. By a suitable
cut-off, $\mathcal{A}_0$ must then contain the constant functions.)
Thus, the algebra $\mathcal{A}_0$ is  contained in $C_0 (X)$. By an easy cut-off argument,
this yields that, in fact, $\ow D_0$ is contained in $C_0 (X)$. Hence, the graph is
uniformly transient. Now, let a harmonic function of finite energy be given. Then, by
the unique solvability of the boundary value problem proven in the previous section,
this function must be a multiple of the constant
function.
\end{proof}

The assumptions of the theorem turn out to be satisfied for a rather
well-known class of examples.

\medskip

\textbf{Example -- $\Z^d$ for $d\geq 3$.} We consider the Euclidean
integer lattice $\Z^d$ with $d\geq 3$ with  standard weights, i.e.,
$c \equiv 0$ and  $b(x,y) = 1$ if  $x$ and $y$ have Euclidean distance one
and $b(x,y) =0$, otherwise. This is obviously a vertex transitive
graph and (since $d\geq 3$) it is transient. Hence, it is uniformly
transient by Corollary~\ref{c:vertrextransitive}. Moreover, it is
folklore that on these lattices any bounded harmonic function is constant and it is well-known (c.f. Theorem~\ref{theorem-harmonic-functions}),
that the absence of non-constant bounded harmonic functions
implies the absence of non-constant harmonic functions of finite energy. Thus, $\Z^d$ does not support non-constant harmonic functions of finite energy.
Therefore, the previous theorem applies and we find that the Royden
compactification of $\Z^d$ is just the one-point compactification.
This is remarkable as the Royden compactification of the one
dimensional integer lattice is far from being the one-point
compactification,  but rather contains a `huge' number of additional
points,  c.f. \cite{Wys}.

\begin{remark} The considerations of the previous example can easily be
adapted to any transient vertex-(quasi)transitive graph with the
Liouville type  property that harmonic functions of finite energy
are constant.
\end{remark}

\section{Spectral theory of uniformly transient graphs}
In this section we consider  some  spectral features of uniformly
transient graphs on $\ell^p$. Our results here generalize the
corresponding results for canonically compactifiable graphs in
\cite{GHKLW} as canonically compactifiable graphs are uniformly
transient.  In fact, in terms of proofs, we basically
adapt the proofs given in \cite{GHKLW}.

\bigskip

Recall that  the semigroup  $e^{-tL}$, $t>0$, and the resolvents
$(L+\al)^{-1}$, $\al>0$, arising from the forms $Q^{(D)}$ associated
to a graph are called \textit{ultracontractive} if they are bounded
operators from $\ell^{2}(X,m)$ to $\ell^{\infty}(X)$.

\begin{lemma}[Uniformly transient graphs are ultracontractive] \label{t:ultracontractive}
Let  $(b,c)$ be a uniformly transient graph over
$X$. Let $m$ be a measure on $X$ of full support and $L = L^{(D)}_m$
the operator associated to $Q^{(D)}_m$. Then, the associated
semigroup and the resolvent are ultracontractive.
\end{lemma}
\begin{proof} We only consider the semigroup operators $e^{-t L}$, $t >0$. The
statements on resolvents can then be derived by standard techniques.
Let $t>0$ be arbitrary. We have
$$e^{-t L} \ell^{2}(X,m)\subseteq D(L) \subseteq D(Q^{(D)}_m)\subseteq C_0 (X) \subseteq \ell^{\infty}(X).$$
Since $e^{-tL}$ is a continuous operator  on $\ell^2 (X,m)$ we now
obtain, by a simple application of the closed graph theorem, that
$e^{-t L} $  can be seen as a continuous map from $\ell^2 (X,m)$ to
$\ell^\infty (X)$. This shows the desired  statement.
\end{proof}

\begin{theorem}[Spectral properties of uniformly transient
graphs]\label{theorem-spectral} Let  $(b,c)$ be a
uniformly transient graph over $X$. Let  $m$ be a measure on $X$ of
full support with $m(X) <\infty$ and let $L = L^{(D)}_m$ be the
operator associated to $Q^{(D)}_m$. Then, the following statements
hold:

\begin{itemize}
\item[(a)]  The operators  $e^{-tL}$, $t>0$, and $(L+\al)^{-1}$, $\al>0$, are trace class.

\item[(b)]  The spectrum of $L$ is purely discrete.

\item[(c)] The infimum of the spectrum of $L$ is bounded below by

$$\alpha :=\frac{1}{C^2 m(X)}, $$ where $C$ is the constant appearing in (i) of Theorem
\ref{main}.

\item[(d)] The  semigroups $e^{-tL}$, $t>0$, and the resolvents
$(L+\al)^{-1}$, $\al>0$, are norm analytic and compact on all
$\ell^{p}(X,m)$, $1\leq p \leq \infty $, and the spectra of the
generators of $e^{-tL}$ on $\ell^{p}(X,m)$ agree for all  $1\leq p
\leq  \infty$.
\end{itemize}
\end{theorem}

\begin{remark} The proof of (a) of this theorem uses a technique
sometimes known as the factorization principle having its roots in
Grothendieck's work \cite{Gro}. For questions of the type considered
here, it has been introduced in \cite{Sto}  to which we refer for
further discussion  (see \cite{BdMSto} as well for further
application in a similar spirit).

\end{remark}

\begin{proof} (a) By  $m(X) <\infty$, there is a canonical continuous  embedding
$$j : \ell^\infty (X)\longrightarrow \ell^2 (X,m), \, f\mapsto f.$$
Thus, by  Lemma  \ref{t:ultracontractive},
$$e^{-t L } = j e^{-t L}$$
is a composition of a continuous maps from $\ell^2 (X,m) $ to
$\ell^\infty (X)$ with a continuous map from $\ell^\infty (X)$ to
$\ell^2 (X,m)$. By  the Grothendieck factorization principle (see
preceding remark)  it is then a Hilbert-Schmidt operator. Then, the
operator
$$e^{-t L }  = e^{- \frac{t}{2}  L} e^{- \frac{t}{2}L}$$
is trace class as it is a  product of two Hilbert-Schmidt operators.

\smallskip

(b) This   follows directly  from (a).

\smallskip

(c) By the definition of $D (Q^{(D)})$ and the closed graph theorem,  the estimate  (ii) of Theorem
\ref{main} holds for all $f\in D (Q^{(D)})$. Thus,  we obtain
directly
$$\|f\|_m^2 \leq m(X) \|f\|_\infty^2 \leq m(X) C^2  Q (f)$$
for all $f\in D (Q^{(D)})$. This easily gives (c).

\smallskip

(d) This  follows directly from  Theorem~2.1.4 and Theorem~2.1.5 of
\cite{Davies}.
\end{proof}

\begin{remark} (a)  Graphs  with  discrete spectrum have been investigated in, e.g.,
\cite{Gol,Kel,KL3, KLW,Woj2}. A general discussion of
characterizations and  perturbation theory of selfadjoint operators
with compact resolvent is recently given in  \cite{LStW}. The
$p$-independence of the spectra of general graph Laplacians has
recently been investigated in \cite{BHK}.

(b) Note that the theorem applies in particular to $\Z^d$ for $d\geq
3$ as this is a uniformly transient graph (as discussed previously).
\end{remark}


We finish this section by giving a lower bound for the eigenvalues of $L$.
\begin{theorem}
 Let  $(b,c)$ be a
uniformly transient graph over $X$. Let  $m$ be a measure on $X$ of
full support with $m(X) <\infty$ and let $L = L^{(D)}_m$ be the
operator associated to $Q^{(D)}_m$. Let $(x_n)$ be an enumeration of $X$. Then, the inequality
 $$\frac{1}{C^2 m(X\setminus \{x_1,\ldots, x_n\})} \leq \lambda_{n+1}(L)$$
 holds, where $C$ is the constant appearing in $\mathrm{(ii)}$ of Theorem \ref{main} and $\lambda_{n}(L)$ is the $n$-th eigenvalue of $L$ counted with multiplicity.
\end{theorem}

\begin{proof}
 To prove the lower bound we use the min-max principle (see e.g. the textbook \cite{Wei}) and the fact that $C_c(X)$ is a form core for $Q^{(D)}_m$ to obtain
 \begin{align*}
\lambda_{n+1}(L) &= \sup_{\varphi_1,\ldots, \varphi_n \in \ell^2(V,m)} \inf_{0\not \equiv \varphi \in C_c(X)\cap\{\varphi_1,\ldots,\varphi_n\}^\perp} \frac{\ow{Q}(\varphi)}{\|\varphi\|^2} \\
  &\geq \inf_{0\not \equiv \varphi \in C_c(X)\, : \, \varphi(x_1) = \ldots = \varphi (x_n)  = 0}  \frac{\ow{Q}(\varphi)}{\|\varphi\|^2} \\
  &\geq \inf_{0\not \equiv \varphi \in C_c(X)\, : \, \varphi(x_1) = \ldots = \varphi (x_n)  = 0} \frac{\|\varphi\|^2_\infty}{C^2 \|\varphi\|^2},
  \end{align*}
where we have used the uniform transience of $(b,c)$ for the last inequality. Now, the statement on the lower bound follows from the elementary fact that bounded functions $\varphi$ that vanish on the set $\{x_1,\ldots,x_n\}$ satisfy
$$\|\varphi\|^2 \leq m(X\setminus \{x_1,\ldots,x_n\})
\|\varphi\|^2_\infty. $$ This finishes the proof.
\end{proof}

\begin{remark}
The best possible lower bound in the above theorem is achieved by
choosing an enumeration $x_1,x_2,\ldots$ of $X$ that satisfies
$m(x_n) \geq m(x_{n+1})$ for each $n \geq 1$. In the case where $m(X) < \infty$
 such an enumeration can always be chosen because $m$ has
to vanish at infinity.
\end{remark}

\begin{appendix}

\section{A characterization of the domain of the form with Dirichlet
boundary conditions} \label{Domain} In this section we provide a proof for Lemma
\ref{Characterization-D-Q-D}, i.e., we show that
$$\overline{C_{c}(X)}^{\aV{\cdot}_{\ow Q}} = \ow D_0
\cap \ell^2 (X,m)$$ whenever $(b,c)$ is a graph over   $X$ and $m$ an measure on $X$ of full support.The statement is a special case of a theorem about general Dirichlet forms.
\begin{proof}
Let $(Q,D)$ be a Dirichlet form  on $L^2(Y,\mu)$ (where $Y$ is a locally compact
Hausdorff space and $\mu$ a Radon measure of full support). One can associate
to $(Q,D)$ the extended Dirichlet space $D_e$, where $u \in D_e$ if and only if there
 exists a $Q$-Cauchy sequence $(u_n)$ with $u_n \to u$ $\mu$-almost surely.

It is well known that the equality  $D = D_e \cap L^2(Y,\mu)$ holds (confer Theorem~1.5.2 in \cite{Fuk}).
In the situation of graphs, i.e., $Q = Q^{(D)}, D = D(Q^{(D)}), Y = X, \mu = m$, it is easy to check
that $D(Q^{(D)})_e = \ow{D}_0 $ (confer Proposition 3.8 in \cite{Schm}).
\end{proof}

\section{A characterization of transience via equivalence of
norms}\label{Transience}  In this section we present a
characterization of transience via an equivalence of norms. This
characterization is probably well-known. As we have not been able to
find it in the literature, we include a proof. We also point out that
it  sheds some additional light on the corresponding equivalence in
our main characterization of ultratranscience.

\medskip

There are various equivalent characterizations of transience. The following definition suits our purposes best. For further details
and a discussion of the relationship to other characterizations  we
refer the reader to \cite{Soa} and \cite{Fuk}.

\begin{definition} A  connected graph $(b,c)$ is \textit{recurrent} if and only if $1$
belongs to $\ow D_0$ and $\ow Q (1) =0$ holds. The graph is called \textit{transient} if it is not recurrent.
\end{definition}

\begin{remark} Obviously, $\ow Q (1) =0$ holds if and only if $c \equiv 0$
holds.

\end{remark}

\begin{theorem}\label{char-transience} Let $(b,c)$ be a connected graph over the countably
infinite $X$. Then, the following assertions are equivalent:
\begin{itemize}
\item[(i)] The graph $(b,c)$ is transient.

\item[(ii)] The norms $\aV{\cdot}_o$ and $\ow Q^{1/2}$ are
equivalent on $C_c (X)$ for every $o \in X$.

\item[(ii$^{\prime}$)]  The norms $\aV{\cdot}_o$ and $\ow Q^{1/2}$ are
equivalent on $C_c (X)$ for one $o \in X$.

\item[(iii)] For every  $o\in X$ there exists $C_o\geq 0$ with
$ |\varphi(o)|\leq C_o \ow Q^{1/2} (\varphi)$
 for all $\varphi \in C_c
(X)$.

\item[(iii$^{\prime}$)] For one   $o\in X$ there exists $C_o\geq 0$ with
$ |\varphi(o)|\leq C_o \ow Q^{1/2} (\varphi)$ for all $\varphi \in C_c
(X)$.

\item[(iv)] $\textup{cap}(o)>0$ for every $o \in X$.

\item[(iv$^\prime$)]  $\textup{cap}(o)>0$ for one $o \in X$.

\end{itemize}

\end{theorem}

\begin{remark} Note that properties (i), (ii), and (iv) of Theorem \ref{main} directly
strengthen properties (i), (iii), and (iv) of the previous theorem.
\end{remark}

\begin{proof} The equivalences between (ii), (iii) and (iv)  and between (ii$^{\prime}$),
(iii$^{\prime}$) and (iv$^{\prime})$ are clear.

\smallskip

(i) $\Longrightarrow$ (ii): It was remarked after the proof of Theorem~\ref{main} that if $\aV{\cdot}$ and $Q^{1/2}$ are not equivalent norms on
$C_c(X)$, then $1 \in \ow{D}_0$ and $c \equiv 0$ so that the graph is recurrent.

\smallskip

(ii) $\Longrightarrow$ (ii$^{\prime}$): This is clear.

\smallskip

(ii$^{\prime}$) $\Longrightarrow$ (i):
 Let $o\in X$ be given such that
$\aV{\cdot}_o$ and $\ow Q^{1/2}$ are equivalent on $C_c (X)$.  Therefore,
there exists $C>0$ such that $\aV{\varphi}_o \leq C \ow Q^{1/2}(\varphi)$ for all $\varphi \in C_c(X)$.
Assume that $(b,c)$ is recurrent. Then, $1$ belongs to $\ow D_0$ and $c \equiv 0$
holds. Hence, there exists a sequence $(\varphi_n)$ in $C_c (X)$
converging to $1$ with respect to $\aV{\cdot}_o$.  In particular, $\lim_{n \to \infty} \ow Q(\varphi_n) = 0$.
As $c \equiv 0$, we then obtain
$$1 =  \aV{1}_o = \lim_{n\to\infty} \aV{\varphi_n}_o \leq \lim_{n\to\infty} C\ow Q^{1/2}(\varphi_n) =0$$
giving a contradiction.
\end{proof}

\begin{corollary}\label{gamma-eq-gamma-o} Let $(b,c)$ be a connected transient graph. Then,
$\gamma$ and $\gamma_o$ are equivalent metrics for any $o\in X$.
\end{corollary}

\section{Harmonic functions}

 In this appendix we discuss the relation of bounded harmonic functions and harmonic functions of finite energy.
 The following theorem is certainly well-known to experts. Due to the lack of a reference we include a proof for the convenience of the reader.

 \begin{theorem} \label{theorem-harmonic-functions}
 Let $(b,0)$ be a transient connected graph over $X$ and suppose there exists a non-constant harmonic function of finite energy on $X$. Then there exists a non-constant bounded harmonic function on $X$.
 \end{theorem}
 \begin{proof}
  Let $f \in \ow{D}$ be harmonic and non-constant and consider the functions $f_n:= (f\wedge n) \vee (-n)$. We use the Royden decompositon (Proposition~\ref{proposition-decomposition}) to obtain bounded $(f_n)_0 \in \ow{D}_0$ and bounded harmonic funcions $(f_n)_h$, such that $f_n = (f_n)_0 + (f_n)_h$. It suffices to show that $(f_n)_h$ is non-constant for some $n$.

  By the cut-off property of $\ow{Q}$ we obtain $\ow{Q}(f_n) \leq \ow{Q}(f)$ for each $n$.
This implies that $(f_n)$ is a bounded sequence in
$(\ow{D},\as{\cdot,\cdot}_o)$ and hence has a weakly convergent
subsequence. Without loss of generality we assume that $(f_n)$
itself converges weakly.  From this and the pointwise convergence of
$f_n$ to $f$ we obtain
  \begin{align*}\ow{Q}(f-f_n) &= \ow{Q}(f) + \ow{Q}(f_n)  - 2 \ow{Q}(f,f_n)\\ &\leq 2(\ow{Q}(f) -  \ow{Q}(f,f_n)) \to 0,\text{ as } n\to \infty.\end{align*}
  Now assume that for each $n$ the function $(f_n)_h$ is constant. Using Lemma~\ref{lem-green formula} we obtain
  $$\ow{Q}(f-f_n) = \ow{Q}(f - (f_n)_h) + \ow{Q}((f_n)_0) \geq \ow{Q}(f - (f_n)_h) = \ow{Q}(f),$$
  where the last equality follows from the fact that $(f_n)_h$ is constant and $c \equiv 0$. Taking the limit $n \to \infty$ shows $\ow{Q}(f) = 0$ which contradicts the assumption that $f$ was not constant. This finishes the proof.
 \end{proof}

\end{appendix}


\end{document}